\documentclass[11pt]{amsart}

\usepackage{amssymb}
\usepackage[all]{xy}

\newtheorem{thm}{Theorem}[section]
\newtheorem{prop}[thm]{Proposition}
\newtheorem{lem}[thm]{Lemma}
\newtheorem{cor}[thm]{Corollary}
\theoremstyle{definition}

\newtheorem{exm}[thm]{Example}
\theoremstyle{remark}

\numberwithin{equation}{section}

		\def\oA{\bar{A}}

\def\tK{\tilde{K}}				\def\tA{\tilde{A}}
\def\tF{\tilde{F}}		\def\tM{\tilde{M}}
\def\tc{\tilde{c}}		

\def\hH{\hat{H}} 

			\def\xx{\times}
\def\*{\otimes}		\def\+{\oplus}		
\def\bop{\bigoplus}	
\def\sb{\subset}         \def\sp{\supset}
\def\spe{\supseteq}      \def\sbe{\subseteq}
\def\0{\emptyset}		\def\8{\infty}

\def\<{\langle}		\def\>{\rangle}

\def\mps{\mapsto}		

\def\Arr{\Rightarrow}

\def\bt{{\bowtie}}

\def\Hom{\mathop\mathrm{Hom}\nolimits}

\def\im{\mathop\mathrm{Im}\nolimits}
\def\End{\mathop\mathrm{End}\nolimits}

\def\rad{\mathop\mathrm{rad}\nolimits}

\def\Mat{\mathop\mathrm{Mat}\nolimits}
\def\rat{\mathop\mathrm{rt}\nolimits}

\def\lat{\mbox{-}\mathrm{lat}}

\def\tr{\mathop\mathrm{tr}\nolimits}

\def\mtr#1{\begin{pmatrix}#1\end{pmatrix}}
\def\smtr#1{\left(\begin{smallmatrix}#1\end{smallmatrix}\right)}
\def\rpp#1#2#3#4#5{\fbox{$\scriptstyle \begin{smallmatrix} #1\\#2	\end{smallmatrix}  \begin{smallmatrix} #3\\#4\\#5\end{smallmatrix}$}}

\def\Bk{B\"ackstr\"om}
\def\oc{one-to-one correspondence}
\def\arq{Auslander-Reiten quiver}
\def\art{Auslander-Reiten transform}
\def\iff{if and only if }

\def\La{\Lambda}      
\def\Ga{\Gamma}       
     
\def\Om{\Omega}       \def\th{\theta}
\def\al{\alpha}       
\def\be{\beta}        \def\eps{\varepsilon}
\def\ga{\gamma}       
      
\def\vi{\varphi}      
\def\si{\sigma}       \def\om{\omega}

 \def\mQ{\mathbb Q}
 
\def\mF{\mathbb F}

 \def\mZ{\mathbb Z}

\def\dA{\mathfrak A} \def\dN{\mathfrak N}

 \def\kT{\mathcal T}

\title[Representations of the alternating group]{Representations and cohomologies of the alternating group of degree 4} 
\author{Yuriy~Drozd and Andriana~Plakosh}
\address{Harvard University, Cambridge, MA, USA and Institute of Mathematics, National Academy of Sciences of Ukraine,
Kyiv, Ukraine}
\email{y.a.drozd@gmail.com}
\address{Institute of Mathematics, National Academy of Sciences of Ukraine, Kyiv, Ukraine}
\email{andrianaplakosh@gmail.com}

\keywords{alternating group, integral representations, Tate cohomologies, Auslander-Reiten quiver, B\"ackstr\"om orders, representations of valued graphs}
\subjclass[2020]{20C10,20C11,20J15,16G30,16G70}

\thanks{The second author was supported by the Simons Foundation grants (1030291, 1290607, A.I.P.).}

\begin{document}
 
 \maketitle
 
\begin{abstract}
 We describe integral representations of the alternating group $\dA_4$, in particular, the \arq\ of its $2$-adic representations.
 Using these results we calculate Tate cohomologies of all $\dA_4$-lattices.
\end{abstract}

 \section*{Introduction}
 
  In the paper \cite{nazA4} Nazarova described $2$-adic representations of the alternating group $\dA_4$ of order $4$.
 Unfortunately, it is very difficult to use this description for other purposes, such as calculation of cohomologies. In this paper
 we propose another approach, analogous to the study of representations and cohomologies of the Klein 4-group in \cite{dpl-kl}.
 Namely, we use the technique of \emph{\Bk\ orders} from \cite{rinrog,rog} and thus relate the description of $2$-adic representations
 of $\dA_4$ with representations of a \emph{valued graph} \cite{DR}. It allows to completely describe the \arq\ of
 $2$-adic representations. As the \art\ in this case coincides with syzygy, it gives almost immediately the values of all
 Tate cohomologies of $2$-adic $\dA_4$-lattices. Since $3$-adic representations of $\dA_4$ are very simple, we also describe 
 all integral representations and their cohomologies. Note that knowing cohomologies is important for applications, such as
 classification of crystallographic and Chernikov groups etc.

\section{Representations. Local structure}
\label{2ad} 
 
 Let $G=\dA_4$ be the alternating group of degree $4$, $N$ be its \emph{Klein subgroup} 
 $N\simeq\{1,a,b,c\mid a^2=b^2=1,\,ab=ba=c\}$, $H=G/K\simeq\<\si\mid\si^3=1\>$. We denote by $A=\mZ G$ the
 group ring of $G$ and set $A_p=A\*\mZ_p$, the $p$-adic completion of $A$, and $QA=A\*\mQ$, the rational
 envelope of $A$. By $A\lat$ (respectively $A_p\lat$) we denote the category of $A$\emph{-lattices}, i.e. $A$-modules
 $M$ such that, as a group, $M$ is a free abelian group of finite rank (respectively, free $\mZ_p$-module of finite
 rank). For an $A$-lattice $M$ we also denote $QM=M\*\mQ$ and $M_p=M\*\mZ_p$. Note that $QM_p\simeq M\*\mQ_p$.
 
 The ring $A$ can be considered as the \emph{crossed product} $K*H$, where $K=\mZ N$ and $H$ naturally acts
 on $K$ by conjugation. Here $QK\simeq\mQ^4$ with the basis $\{e_1,e_a,e_b,e_c\}$, where
  \begin{align*}
   e_1 &= \frac{1+a+b+c}4,\\ 
   e_a &= \frac{1+a-b-c}4,\\
   e_b &= \frac{1-a+b-c}4,\\
   e_c &= \frac{1-a-b+c}4.         
  \end{align*}
 Under this identification, $K_p=\mZ_p^4$ if $p\ne2$ and $K_2$ embeds into $\mZ_2^4$ so that $a,b,c$
 identifies, respectively, with the elements $(1,1,-1,-1),(1,-1,1,-1)$ and $(1,-1,-1,1)$. 
  The action of $H$ is trivial on the first component of $QK$ and cyclically permutes the other three. Hence 
  \begin{align*}
   &QA\simeq QK*H\simeq \mQ H\xx\Mat(3,\mQ)\simeq\mQ\xx\mQ[\th]\xx\Mat(3,\mQ),\\
   \intertext{where $\th$ is a primitive cubic root of 1. If $p\notin\{2,3\},$ then}   
   &A_p\simeq\mZ_p\xx\mZ_p[\th]\xx\Mat(3,\mQ_p)\ (\text{a maximal order in } QA_p),\\
   \intertext{and}
   & A_3\simeq \mZ_3H\xx\Mat(3,\mZ_3).
  \end{align*}
  Therefore, all indecomposable $A_p$-lattices for $p\notin\{2,3\}$ are irreducible lattices $\mZ_p,\,Z_p[\th]$ and $I_p=\mZ_p^3$,
  and for $A_3$ there is one more indecomposable lattice $\mZ_3H$. 
   
  The case $p=2$ is quite different, since $\mZ_2K$ is no more a maximal order.   
  Recall that every group ring $R$ is \emph{Gorenstein}, i.e. $\mathop\mathrm{inj.dim}_RR=1$. Therefore, all non-projective
  $R_p$-lattices are actually lattices over the overring $R^+=\End_R(\rad R_p)$. As $\#(H)$ is invertible in $\mZ_2$, for the crossed
  product $A_2=K_2*H$ we have that $\rad A_2=(\rad K)*H$ and $A^+=K^+*H$. Note that $K^+$ is a \emph{\Bk\ order} 
  in the sense of \cite{rinrog}. It means that there is a hereditary order $\tK$ such that $\tK\sp K^+\sp\rad\tK=\rad K^+$. In our
  case $\tK=\mZ^4_2$ and $K^+=\{(x_1,x_2,x_4,x_4)\mid x_1\equiv x_2\equiv x_3\equiv x_4\!\pmod2\}$. As $\tA=\tK*H$ is
  also hereditary, $A^+$ is also a \Bk\ order. Namely, 
  $$ \tA\simeq \mZ_2\+\mZ_2[\th]\xx\Mat(3,\mZ_2).$$
  One can easily see that $A^+$ embeds into $\tA$ as the subring of triples $(x_1,x_2,x_3)$, where 
  $x_1\in\mZ,\,x_2\in\mZ[\th]$ and $x_3=(\xi_{ij})\in\Mat(3,\mZ_2)$, such that $x_1\equiv \xi_{11}\!\pmod2$,
  $\xi_{12}\equiv\xi_{13}\equiv\xi_{21}\equiv\xi_{31}\equiv0\!\pmod 2$
  and $\rho(x_2)\equiv\smtr{\xi_{22}&\xi_{23}\\\xi_{32}&\xi_{33}}\!\pmod 2$, where $\rho$ denotes the regular representation of
  $\mZ_2[\th]$: $\rho(u+v\th)=\smtr{u&-v\\v&u-v}$. We denote by $L_1,L_2,L_3$ the irreducible $\tA$-lattices belonging,
  respectively, to the components $\mZ_2,\,\mZ_2[\th]$ and $\Mat(3,\mZ_2)$, and by $P_1$ and $P_2$ the indecomposable
  projective $A^+$-lattices $P_i=A^+e_i$, where $e_1=\frac{1+\si+\si^2}{3}$ and $e_2=1-e_1$. Note that the only 
  indecomposable $A_2$-lattices that are not $A^+$-lattices are $B_i=A_2e_i$. They are \emph{bijective}, i.e. both
  projective and injective in the exact category $A_2\lat$.
   
  Recall \cite{rinrog}, that representations of a \Bk\ order are classified by representations of a \emph{weighted graph} $\Ga$ 
  in the sense of \cite{DR}. Namely, the vertices of $\Ga$ are in \oc\ with the simple components of semsimple algebras
  $\oA=A^+/\rad A^+$ and $\oA'=\tA/\rad\tA$. In our case
   \[
     \oA=(K^+/\rad K^+)*H\simeq\mF_2 H\simeq\mF_2\xx\mF_4
    \]
    and
   \[
    \oA'\simeq(\tK/\rad\tK)*H\simeq\mF_2*H\xx\Mat(3,\mF_2)\simeq\mF_2\xx\mF_4\xx\Mat(3,\mF_2).
   \]
  Hence, the corresponding graph $\Ga$ (with orientation) is of type $\tF_{41}$:
  \[
   \Ga: \quad  \vcenter
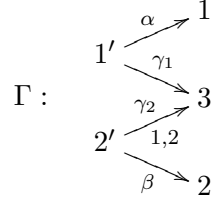
{\xymatrix@R=.5ex{ & 1 \\ 1' \ar[ur]^\al \ar[dr]^{\ga_1} \\&3\\ 2' \ar[ur]^{\ga_2}_{1,2}\ar[dr]_\be\\ &2    }}
  \]
  Here $1,2,3$ correspond, respectively, to the components $\mF_2,\mF_4$ and $\Mat(3,\mF_2)$ of $\oA'$, while
  $1'$ and $2'$ correspond, respectively, to the components $\mF_2$ and $\mF_4$ of $\oA$. The weights of all arrows
  except $\ga_2$ are $(1,1)$, so we do note write them. In representations of $\Ga$ the arrows $\al,\ga_1$ correspond to
  matrices with entries from $\mF_2$, the arrows $\ga_2,\be$ correspond to matrices with entries from $\mF_4$.
  The vector dimension of a representation $M$ of $\Ga$ we denote by $\rpp{d_{1'}}{d_{2'}}{d_1}{d_3}{d_2}$. Recall that
  the $A^+$-lattice $M$ corresponding to a representation $V$ of this graph is the preimage in 
  $\tM=L_1^{d_1}\+L_3^{d_3}\+L_2^{d_2}$ of $\im\vi(V)$, where $\vi(V):V(1')\+V(2')\to V(1)\+V(3)\+V(2)$ is given by the 
  matrix
  \[
   \mtr{V(\al)&0\\V(\ga_1)&V(\ga_2)\\0&V(\be)}.
  \]
  and we identify $V(1)\+V(2)\+V(3)$ with $\tM/2\tM$.
  The Auslander-Reiten quiver of the category of representations of graph $\Ga$ consists of preprojective, preinjective
  and regular components. When we pass to representations of the \Bk\ order, we have to glue the preprojective and
  preinjective components into one component (we call it \emph{principal}) \cite{rog}. Namely, we omit simple injective 
  and simple projective modules and then add arrows from the remaining injective to the remaining projective modules. 
  As a result, the principal component for the order $A^+$ becomes:
  {\small
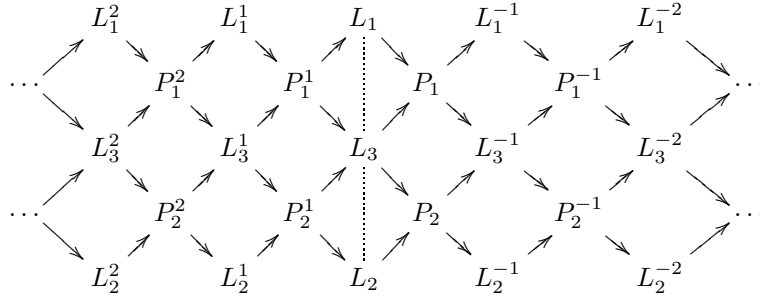
 \[
  \xymatrix@C=.65em@R=.7em{& L^2_1 \ar[dr] &&  L^1_1 \ar[dr] && L_1\ar[dr] \ar@{.}[dd]&&  L_1^{-1} \ar[dr] && L_1^{-2}\ar[dr]  \\
 *+{\dots\ \ } auto.\ar[ur]\ar[dr] && P^2_1 \ar[dr] \ar[ur] &&  P^1_1 \ar[dr] \ar[ur] &&  P_1 \ar[dr] \ar[ur] &&  P^{-1}_1 \ar[dr] \ar[ur] 
 && *+{\ \ \dots}\\
  & L^2_3 \ar[dr] \ar[ur] &&  L^1_3 \ar[dr] \ar[ur] && L_3\ar[dr] \ar[ur] \ar@{.}[dd] &&  L_3^{-1} \ar[dr] \ar[ur] && L_3^{-2}\ar[dr] \ar[ur] \\
 *+{\dots\ \ }\ar[ur]\ar[dr] && P^2_2 \ar[dr] \ar[ur] &&  P^1_2 \ar[dr] \ar[ur] &&  P_2 \ar[dr] \ar[ur] &&  {P^{-1}_2} \ar[dr] \ar[ur]
  && *+{\ \ \dots}ahlength\\  
 & L^2_2 \ar[ur] &&  L^1_2 \ar[ur] && L_2\ar[ur] &&  {L_2^{-1}} \ar[ur] && {L_2^{-2}} \ar[ur]  }
 \]}%
  Here by $M^k$ we denote the $k$-th Auslander-Reiten transform $\tau^kM$ of the lattice $M$ over the ring $A_2$. 
 As shown in \cite{rejection}, $M^k\simeq\Om^kM$, the $k$-th syzyzy of $M$ as of $A_2$-module, and $P_i^1$ are just the 
 \emph{injective $A^+$-lattices}, i.e. those dual to projective ones, or, the same, injective in the exact category $A^+\lat$. 
 Actually, $P_i$ is the unique minimal overmodule and $P_i^1$ is the unique maximal submodule of $B_i=A_2e_i$. Recall 
 that $B_i$ are the only indecomposable $A_2$-lattices which are not $A^+$-lattices. 
 They are \emph{bijective}, i.e. both projective and injective in $A_2\lat$.
 Note also that the other irreducible $A^+$-lattices are $L^{\pm1}_1$ (just as $L_3$ they belong to the component $\Mat(3,\mZ)$).
 
  The dimensions of the corresponding representations of $\Ga$ are given in the next diagram.
     {\small
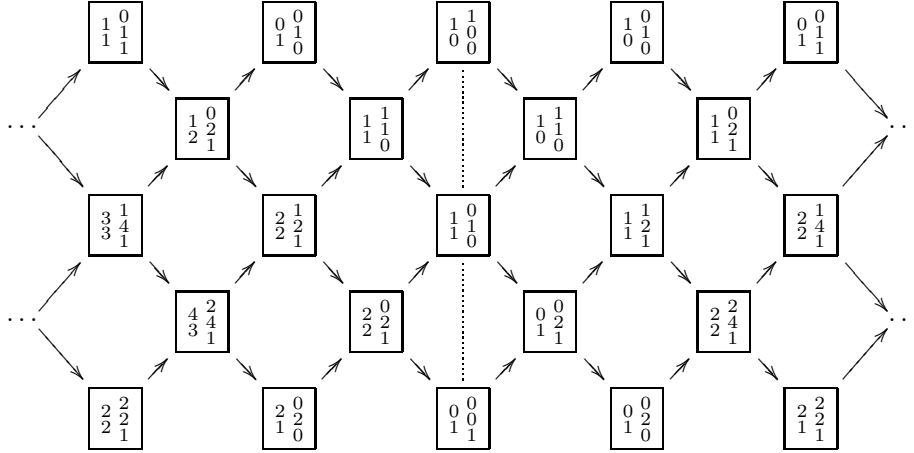
 \[
  \xymatrix@C=.65em@R=.7em{& \rpp11011 \ar[dr] &&  \rpp01010 \ar[dr] && \rpp10100  \ar@{.}[dd] \ar[dr] &&  \rpp10010 \ar[dr]
   && \rpp01011 \ar[dr]  \\
 *+{\dots\ \ } \ar[ur]\ar[dr] && \rpp12021 \ar[dr] \ar[ur] &&  \rpp11110 \ar[dr] \ar[ur] &&  \rpp10110 \ar[dr] \ar[ur] && \rpp11021 \ar[dr] \ar[ur] 
 && *+{\ \ \dots}\\
  &\rpp33141 \ar[dr] \ar[ur] &&  \rpp22121 \ar[dr] \ar[ur] && \rpp11010  \ar@{.}[dd] \ar[dr] \ar[ur] &&  \rpp11121 \ar[dr] \ar[ur] && 
  \rpp22141 \ar[dr] \ar[ur] \\
 *+{\dots\ \ }\ar[ur]\ar[dr] && \rpp43241 \ar[dr] \ar[ur] && \rpp22021 \ar[dr] \ar[ur] && \rpp01021 \ar[dr] \ar[ur] && \rpp22241 \ar[dr] \ar[ur]
  && *+{\ \ \dots}\\  
 & \rpp22221 \ar[ur] &&  \rpp21020 \ar[ur] && \rpp01001 \ar[ur] && \rpp01020 \ar[ur] &&  \rpp21221 \ar[ur]  }
 \]}%
 Note that the symmetry with respect to the central (dotted) column corresponds to the duality 
 $M\mps M^*=\Hom_{\mZ_2}(M,\mZ_2)$ in the category of $A_2$-lattices.
 
 The regular components for $A^+$-lattices are the same as those for the graph $\Ga$. They consist of \emph{homogeneous
 tubes} $\kT^f$ corresponding to monic irreducible polynomials from $\mF_2[t]$, except $t-1$ and $t^2+t+1$, and two \emph{special
 tubes} $\kT^1$ and $\kT^\th$. The homogeneous tubes are of the form
 \[
   \xymatrix@1{ T^f_1 \ar@/^1ex/[r] \ar@/_1ex/[r] & T^f_2 \ar@/^1ex/[r] \ar@/_1ex/[r] &  T^f_3 \ar@/^1ex/[r] \ar@/_1ex/[r] 
   & *+{\ \ \dots}}.
 \]
  In these components the Auslander-Reiten transform (or, the same, the syzygy) acts trivially.  
 The dimension of the representation of $\Ga$ corresponding to $T^f_k$ is $\rpp{2kd}{2kd}{kd}{3kd}{kd}$, 
 where $d=\deg f$. 
 
 The special tubes are of the forms:
  \begin{align*}
   & \vcenter{ \xymatrix@R=1em{ T^1_{11} \ar[r] & T^1_{12} \ar[r]\ar[dl] & T^1_{13} \ar[r]\ar[dl] & *+{\ \dots} \ar[dl] \\
     				 T^1_{21} \ar[r] & T^1_{22} \ar[r]\ar[ul] & T^1_{23} \ar[r]\ar[ul] & *+{\ \dots} \ar[ul] }} \\
   \intertext{and}
   & \vcenter{ \xymatrix@R=1em{ T^\th_{11} \ar[r] & T^\th_{12} \ar[r]\ar[dl] & T^\th_{13} \ar[r]\ar[dl]  & *+{\ \dots} \ar[dl]  \\ 
   									T^\th_{21} \ar[r] & T^\th_{22} \ar[r]\ar[dl] & T^\th_{13} \ar[r]\ar[dl]  & *+{\ \dots} \ar[dl]  \\
   									T^\th_{21} \ar[r] & T^\th_{22} \ar[r]\ar@/_2pt/[uul] & T^\th_{13} \ar[r]\ar@/_2pt/[uul]  & *+{\ \dots} \ar@/_2pt/[uul]  } }  				 
  \end{align*} 
  The dimensions of the corresponding representations of $\Ga$ are: \label{tub} 
  \begin{align*}
  &  \rpp{2k}{2k}{k}{3k}{k} &&\hspace*{-7em} \text{for } T^1_{i,2k}\ \,(k>0)\\
  &  \rpp{1}{1}{1}{1}{1}+\rpp{2k}{2k}{k}{3k}{k} &&\hspace*{-7em} \text{for }  T^1_{1,2k+1}\\
  &  \rpp{1}{1}{0}{2}{1}+\rpp{2k}{2k}{k}{3k}{k} &&\hspace*{-7em} \text{for }  T^1_{2,2k+1}\\
  &  \rpp{4k}{4k}{2k}{6k}{2k} &&\hspace*{-7em} \text{for }  T^\th_{i,3k}\ \,(k>0)\\
  &  \rpp{0}{2}{0}{2}{1}+\rpp{4k}{4k}{2k}{6k}{2k} &&\hspace*{-7em} \text{for }  T^\th_{1,3k+1}\\  
  &  \rpp{2}{1}{2}{2}{0}+\rpp{4k}{4k}{2k}{6k}{2k} &&\hspace*{-7em} \text{for }  T^\th_{2,3k+1}\\  
  &  \rpp{2}{1}{0}{2}{1}+\rpp{4k}{4k}{2k}{6k}{2k} &&\hspace*{-7em} \text{for }  T^\th_{3,3k+1}\\  
  &  \rpp{2}{3}{2}{4}{1}+\rpp{4k}{4k}{2k}{6k}{2k} &&\hspace*{-7em} \text{for } T^\th_{1,3k+2}\\  
  &  \rpp{4}{2}{2}{4}{1}+\rpp{4k}{4k}{2k}{6k}{2k} &&\hspace*{-7em} \text{for }  T^\th_{2,3k+2}\\  
  &  \rpp{2}{3}{0}{4}{2}+\rpp{4k}{4k}{2k}{6k}{2k} &&\hspace*{-7em} \text{for }  T^\th_{3,3k+2}.  
  \end{align*}
  The Auslander-Reiten transform (or, the same, syzyzy) acts as follows:
   \begin{align*}
    &  \xymatrix@1{ T^1_{1k} \ar@/^1ex/[rr] && T^1_{2k} \ar@/^1ex/[ll] } \\
    &  \xymatrix@1{ T^\th_{1k} \ar[rr] && T^\th_{2k} \ar[rr] && T^\th_{3k} \ar@/^2ex/[llll].  }
   \end{align*}

\smallskip
\section{Globalization}
\label{glob} 

 To describe indecomposable $A$-lattices, we use the following results of \cite{F-Gen}.
 
  \begin{prop}\label{gl-1} 
 Let $M(p)$ be $A_p$-lattices given for all prime $p$.
 \begin{enumerate}
 \item  There is an $A$-lattice $M$ such that $M_p\simeq M(p)$ for all $p$  \iff there is a $QA$-module
 $V$ such that $QM(p)\simeq\mQ_p\*_\mQ V$ for all $p$. Then we say that all $M(p)$ are \emph{of the
 same rational type}.
 \item  Such lattice $M$ is decomposable  \iff there are direct summands $N(p)$ of every $M(p)$ such
 that all $N(p)$ are of the same rational type. In particular, if $M'$ is another lattice with the same
 localizations, $M$ and $M'$ decomposes simultaneously.
 \end{enumerate}
 \end{prop}
 
 They say that two $A$-lattices $M$ and $M'$ are \emph{of the same genus} if $M_p\simeq M'_p$ for
 all $p$. As $A\sb 6\tA$, the following result follows from \cite[Thm.\,3.7]{adeles}.
 
  \begin{prop}\label{gl-2} 
  If two $A$-lattices belong to the same genus, they are isomorphic.
  \end{prop}
 
 Note that if $p\notin\{2,3\}$, for every $QA$-module $V$ there is a unique $A_p$-lattice $L$ such that
 $QL\simeq\mQ_p\*_\mQ V$. Therefore, an $A$-lattice is completely defined by its $2$-adic and
 $3$-adic localizations. If $p\in\{2,3\}$, every $QA_p$-module is of the form $\mQ_p\*_\mQ V$, where 
 $V$ is a $QA$-module. $V$ decomposes as $V\simeq\mQ^{r_1}\+\mQ[\th]^{r_2}\+W^{r_3}$, 
 where $W$ is the unique simple $\Mat(3,\mQ)$-module. We write $\rat M_p=(r_1,r_2,r_3)$ and call 
 $\rat M_p$ the \emph{rational type} of $M_p$. Hence an $A$-lattice is defined by a pair
 $M_2,\,M_3$ of lattices over $A_2$ and $A_3$ which are of the same rational type. 

 Note that the unique indecomposable $A_3$-lattice which is not irreducible is the lattice $\La=\mZ_3H$.
 The rational type of $\La$ is $(1,1,0)$. From now on, let $M$ be an $A$-lattice of rational type 
 $(d_1,d_2,d_3)$ and $M_3=\mZ_3^{k_1}\+\mZ_3[\th]^{k_2}\+L_3^{k_3}\+\La^k$, where $L_3$ is the 
 irreducible $\Mat(3,\mZ_3)$-lattice. Note that the dimension of the corresponding representation 
 of the valued graph $\Ga$ is $\rpp{d_{1'}}{d_{2'}}{d_1}{d_3}{d_2}$ for some $d_{1'},d_{2'}$. 
 Proposition~\ref{gl-1} means that $\,k_1+k=d_1,\,k_2+k=d_2$ and $k_3=d_3$. 
 It implies a description of $A$-lattices $M$ such that $M_2$ is indecomposable.
 
  \begin{thm}\label{p0} 
  Let $N$ be an indecomposable $A_2$-lattice, $\rat N=(c_1,c_2,c_3)$ and $\tc=\min(c_1,c_2)$.
  Denote by $N^k\ (0\le k\le \tc)$ the $A$-lattice such that $N^k_2\simeq N$ and 
  $N^k_3\simeq \La^k\+\mZ_3^{c_1-k}\+\mZ_3[\th]^{c_2-k}+L_3^{c_3}$. Every $A$-lattice $M$
  such that $M_2\simeq N$ is isomorphic to one of $N^k$.
  \end{thm}
 
 Let now $M_2\simeq \bop_{i=1}^s N^i$, where $s>1$, $\rat N^i=(c_{1i},c_{2i},c_{3i})$ and 
 $\tc_i=\min(c_{1i},c_{2i})$. The following result is obvious.
 
  \begin{prop}\label{p1} 
  If $k\le\sum_{i=1}^s\tc_i$, then $M$ decomposes as $\bop_{i=1}^s M^i$, where $M^i_2\simeq N^i$
  and $M^i_3\simeq \La^k\+\mZ_3^{c_{1i}-b_i}\+\mZ_3[\th]^{c_{2i}-b_i}\+L_3^{c_{3i}}$, where $b_i$
  are arbitrary integers such that $b_i\le \tc_i$ and $\sum_{i=1}^sb_i=k$.
  \end{prop}

  Thus from now on we suppose that $k>\sum_{i=1}^s\tc_i$. 
  
   \begin{prop}\label{p2} 
   If $N$ is a direct summand of $M_2$ of rational type $(c,c,c_3)$, then $M$ has a direct summand $M'$ 
   such that $M'_2\simeq N$ and $M'_3\simeq \La^c\+L_3^{c_3}$. 
   \end{prop}
   
   Therefore, if $M$ is indecomposable, $M_2$ has no proper direct summands of rational type $(c,c,c_3)$.
   In what follows we always suppose that this condition is satisfied.   
  
  \begin{prop}\label{p3} 
  If $c_{1i}< c_{2i}$ for all $i$ or $c_{1i}>c_{2i}$ for all $i$, then $M$ decomposes.
   \end{prop} 
    \begin{proof}
    Let $c_{1i}<c_{2i}$ for all $i$ and $c_{11}$ is the minimal among $c_{1i}$. Then $M$ has a direct
    summand $M'$ such that $M'_2\simeq N_1$ and 
    $M'_3\simeq \La^{c_{11}}\+\mZ_3[\th]^{c_{21}-c_{11}}\+L_3^{c_{13}}$.
    \end{proof}
    
    Propositions~\ref{p1}-\ref{p3} imply a description of indecomposable $A$-lattices $M$ such that $M_2$
    has two indecomposable components.
    
     \begin{prop}\label{p4} 
     If $M$ is indecomposable and $s=2$, then, up to permutation of $N_1$ and $N_2$, $c_{11}<c_{21}$ 
     and $c_{12}>c_{22}$. There are $c^+-\tc$ lattices, where $\tc=\tc_1+\tc_2$ and
     $c^+=\min(c_{11}+c_{12},c_{21}+c_{22})$ corresponding to decompositions
     $M_3\simeq\La^k\+\mZ_3^{c_{11}+c_{12}-k}\+\mZ_3[\th]^{c_{21}+c_{22}-k}\+L_3^{c_{31}+c_{32}}$,
     where $\tc<k\le c^+$.
     \end{prop}
 
    The description of $2$-adic lattices shows that $|c_{1i}-c_{2i}|\le2$. If all $|c_{1i}-c_{2i}|=1$, 
    Propostion~\ref{p2} implies that $M$ contains a direct summand $M'$ such that 
    $M'_2\simeq N_i\+N_j$ for some $i\ne j$. The same holds if $c_{1i}-c_{2i}=c_{2j}-c_{1j}=2$.
    Hence we can suppose now that there is one $i$ such that $c_{1i}-c_{2i}=2$ and $c_{2j}-c_{1j}=1$
    for all $j\ne i$ (or vice versa). Then Proposition~\ref{p2} implies that if $M$ is indecomposable, 
    there are at most two such indices $j$. One immediately see that the unique possibility with 
    two such indices is when $M_3\simeq\La^k\+L_3^{c_3}$, where $k=\tc_1+\tc_2+\tc_3+2$.
   
   Thus, we have described all indecomposable $A$-lattices $M$ with decomposable $M_2$.
   
    \begin{thm}\label{p5}   
    Denote by
    \begin{itemize}
    \item $\dN^1$ be the set of indecomposable $A_2$-lattices such that $c_1-c_2=1$,
    \item $\dN^2$ be the set of indecomposable $A_2$-lattices such that $c_1-c_2=2$,
    \item $\dN_1$ be the set of indecomposable $A_2$-lattices such that $c_2-c_1=1$,
    \item $\dN_2$ be the set of indecomposable $A_2$-lattices such that $c_2-c_1=2$.
    \end{itemize}
    Then the only possibilities for indecomposable $A$-lattices $M$ such that $M_2$ is
    decomposable are the following:
    \begin{enumerate}
    \item $M_2\simeq N_1\+N_2$, where $N_1\in\dN^1\cup\dN^2,\,N_2\in \dN_1\cup\dN_2$, 
    $M_3\simeq\La^{\tc_1+\tc_2+1}$. We denote such $M$ by $N_1\bt N_2$.
    \item $M_2\simeq N_1\+N_2$, where $N_1\in\dN^2,\,N_2\in \dN_2$, 
    $M_3\simeq\La^{\tc_1+\tc_2+2}$. We denote such $M$ by $N_1\bt^2 N_2$.
    \item $M_2\simeq N_1\+N_2\+N_3$, where $N_1\in\dN^2,\,N_2,N_3\in\dN_1$ or
    $N_1\in\dN_2,\,N_2,N_3\in\dN^1$, $M_3\simeq\La^{\tc_1+\tc_2+\tc_3+2}$. We denote such $M$
    by $N_1\bt(N_2\+N_3)$.
    \end{enumerate}
    All lattices described in (1-3) are indecomposable and pairwise nonisomorphic.
    \end{thm}
    
    Theorems~\ref{p0} and \ref{p5} give a complete description of indecomposable $A$-lattices.

    Note that the decomposition of an $A$-lattice $M$ into a direct sum of indecomposables is far from
    being unique. 
    
     \begin{exm}\label{gl-2} 
      \begin{enumerate}
      \item  If $N_1,N_1'\in\dN^1\cup\dN^2$ and $N_2,N_2'\in\dN_1\cup\dN_2$, then 
      $N_1\bt N_2\+N_1'\bt N_2'\simeq N_1\bt N_2'\+N_1'\bt N_2$.
      
      \item  Let $N_1\in\dN^2,\,N_2,N_3\in\dN_1,\,N_1'\in\dN_2$ and $N'_2,N'_3\in\dN^1$. Then
      $N_1\bt(N_2\+N_3)\+N'_1\bt(N'_2\+N'_3)\simeq 
      N_1\bt^2N_1'\+N_2\bt N'_2\+N_3\bt N'_3$. 
      Hence even the number of indecomposable summands can differ in different decompositions.
      \end{enumerate}
     \end{exm}

\section{Cohomology}
\label{cohom}

   We are going to calculate Tate cohomologies of $G$-lattices. As $\#(G)=12$, for every $G$-module $M$ the groups
   $\hH^n(G,M)$ split into their $2$-components $\hH^n(G,M)_2$ and $3$-compoinents $\hH^n(G,M)_3$. Moreover, if
   $M$ is a lattice, $\hH^n(G,M)_p\simeq\hH^n(G,M_p)$. So we can consider $2$-adic and $3$-adic cases separately.

  For the group $G=\dA_4$ the spectral sequence $E_2^{pq}=H^p(H,H^q(N,M))\Arr H^n(G,M)$ degenerates
  both in $2$-adic and in $3$-adic case. Namely,
  for $2$-adic lattices $E_2^{pq}=0$ if $p\ne0$. So we obtain isomorphisms
  \[
   H^n(G,M)\simeq H^n(N,M)^H.
  \]  
 For $3$-adic lattices $E_2^{pq}=0$ if $q\ne0$, hence
 \[
  H^n(G,M)\simeq H^n(C,M^N).
 \] 
 In the $3$-adic case we have indecomposable lattices $\mZ,\mZ[\eps],\mZ H$ and $I_3$. Note that $K$ acts
 trivially on $\mZ_3,\mZ_3[\eps]$ and $\mZ_3H$ and has no fixed points on $I_3$. The quotient $H$ is cyclic, so
 its cohomologies are periodic with period $2$. Easy calculations give:
  \begin{align*}
   \hH^n(G,\mZ_3)=
    \begin{cases}
     \mF_3 &\text{if $n$ is even},\\
     0 &\text{if $n$ is odd};
    \end{cases}\\
      \hH^n(G,\mZ_3[\eps])=
    \begin{cases}
     0 &\text{if $n$ is even},\\
     \mF_3 &\text{if $n$ is odd}.
    \end{cases}
  \end{align*}
  The other indecomposable lattices are projective, hence have trivial Tits cohomologies.
   
   For $2$-adic lattices we use the following result analogous to \cite[Lem.\,2.2]{dpl-kl} and with analogous proof. 
   \begin{lem}\label{co-1} 
    Let $M$ be an indecomposable $A^+$-lattice corresponding to the representation $V$ of the quiver $\Ga$ of dimension
    $\rpp{d_{1'}}{d_{2'}}{d_1}{d_3}{d_2}$. If $M\not\simeq L_1$, then $\hH^0(G,M)\simeq\mF_2^{d_1}$.
   \end{lem}
    \begin{proof}
     Recall that $\hH^0(G,M)=M^G/tr M$, where $M^G$ is the set of invariants: 
     $M^G=\{m\in M\mid gm=m \text{ for all } g\in G\}$, and $tr=\sum_{g\in G}g$. If we consider $M$ as a sublattice
     between $\tM=\tA M= L_1^{d_1}\+L_2^{d_2}\+L_3^{d_3}$ and 
     $\rad\tM=2\tM=2L_1^{d_1}\+2L_2^{d_2}\+2L_3^{d_3}$. Then $\tM^G=L_1^{d_1}$ and $M^G=\tM^G\cap M\spe 2L_1^{d_1}$. 
     Let $\pi:\tM\to\tM^G$ be the projection. If $u\in M^G$ and $\pi(u)\notin2\tM^G$, then $\tM^G=U\+N$, where $U=\mZ_2u$.
     Combining the restriction of $\pi$ onto $M$ with the projection $\tM^G\to U$, we obtain a homomorphism $\eta:M\to U$
     such that $\eta\eps=1_U$, where $\eps:U\to M$ is the embedding. Therefore, $U$ is a direct summand of $M$ and 
     $M\simeq U\simeq L_1$, which is impossible. Hence $M^G=2\tM_G$. On the other hand, $\pi(M)=\pi(\tM)$, 
     since the projection of 
     $A^+$ onto the first component of $QA_2$ is maximal. Therefore $\tr M=\tr\tM=\tr\tM^G=12\tM^G=2 M^G$, 
     since $\tr\tM\sbe\tM^G$. Thus $\hH^0(G,M)= M^G/2M^G\simeq\mF_2^{d_1}$.
    \end{proof}   
    
    Note that the rational types of the lattices $M$ and $M^*$ are equal, hence $\hH^0(M)\simeq\hH^0(M^*)$.
    As also $H^n(G,M^*)\simeq H^{-n}(G,M)$ (\cite[Prop.\,3.2]{dpl1}), one immediately obtains by an obvious
    induction the following corollary.
    
     \begin{cor}\label{co-2} 
      The groups $\hH^n(M)$ do not chane when one replaces $n$ by $-n$ or $M$ by $M^*$.
      \end{cor} 
 
    Having the Auslander-Reiten quiver, we only need to know $\hH^0(G,M)$ for all indecomposable $M$,
    since $\hH^n(G,M)\simeq\hH^0(G,\tau^nM)$, as $\tau M=\Om M$. Actually, for every representation $V$ from 
    the preinjective component of the Auslander-Reiten quiver there is a number $m|6$ such that 
    $\dim \tau^kM=\dim M+q\om$, where $\om=\rpp22131$ and $q\in\{1,2\}$. Therefore, the value of $d_1$ just changes
    by $q$. It gives a simple procedure for calculation of cohomologies of lattices from the principal component. Here is
    the result of these calculations. 

     \begin{thm}\label{main1} 
     Let $M$ be an idecomposable $A^+$-lattice from the principal component, namely, $M=M_0^r$, where
     $M_0\in\{L_1,L_2,L_3,P_1,P_2\}$. Set $k=[|n+r|/m],\ i=|n+r|-km$. Then $\hH^{n}(G,M)\simeq\mF_2^{qk+r_i}$, 
     except the case when $M_0=L_1$ and $i=0$. The values $m,q$ and $r_i$ 
     depend on $M_0$. Namely:     

  If $M_0=L_1$, then  $m=6$, $q=1$ and
  $$\begin{array}{|c|c|c|c|c|c|}
  \hline
   i & 1&2&3&4&5 \\
   \hline
   r_i & 0&0&1&1&0\\
   \hline
   \end{array} $$

 If $M_0=L_2$, then  $m=6$, $q=2$ and
  $$\begin{array}{|c|c|c|c|c|c|c|}
  \hline
   i & 0&1&2&3&4&5 \\
   \hline
   r_i & 0&0&2&0&2&2\\
   \hline
   \end{array} $$ 
  
  If $M_0=L_3$, then  $m=2$, $q=1$ and $r=i$.
  
  If $M_0=P_1$, then  $m=3$, $q=1$ and
  $$\begin{array}{|c|c|c|c|c|c|c|}
  \hline
   i & 1&2&3 \\
   \hline
   r & 1&0&1\\
   \hline
   \end{array} $$

 If $M_0=P_2$, then  $m=3$, $q=2$ and   
  $$\begin{array}{|c|c|c|c|c|c|c|}
  \hline
   i & 1&2&3 \\
   \hline
   r & 0&2&2\\
   \hline
   \end{array} $$
   
   If $M\simeq L_1^r$ and $6|(n+r)$, then $\hH^n(M)\simeq\mZ/4\mZ$.
       \end{thm} 
   
   For representations from tubes the situation is even easier, since the values of $d_1$, hence of $\hH^0$, 
    are given on page~\pageref{tub} after the description of tubes, and we know the action of $\tau$. So we obtain
    the following result.    
     
      \begin{thm}\label{main2} 
           \begin{align*}
      & \hH^n(G,T^f_k)\simeq\hH^0(G,T^f_k)\simeq\mF_2^{kd}, \text{ \rm where } d=\deg f.\\
      & \hH^{2n+r}(G,T^1_{i,2k+j})\simeq\hH^0(G,T^1_{i' ,j})\simeq\mF^{k+c}, \text{ \rm where } \\
      & \qquad c= \begin{cases}
        1 &\text{\rm if }  j=i'=1,\\
        0 &\text{\rm if $j=0$ or $i'=0$}
       \end{cases}\\
      & \qquad  (i'\equiv i+r\hskip-6pt\pmod2, \text{ \rm and } i'\in\{0,1\}).\\
      & \hH^{3n+r}(G,T^\th_{i,3k+j})\simeq\hH^0(G,T^\th_{i',j})\simeq\mF_2^{2k+c}, \text{ \rm where }\\
      & \qquad   c= \begin{cases}
        2 &\text{\rm if $j=1,\,i'=2$ or } j=2,\,i'\ne0,\\
        0 &\text{\rm otherwise}
       \end{cases}\\
     &  \qquad  (i'\equiv i+r\hskip-6pt\pmod3, \text{ \rm and } j'\in\{0,1,2\}).
     \end{align*}     
      \end{thm}
      
      \medskip
    
    \subsection*{Competing Interests}
    
    Authors confirm that there are no financial or non-financial interests that are 
    directly or indirectly related to the work submitted for publication.

\end{document}